\documentclass[14pt,a4paper]{article}

\usepackage{mathrsfs}
\usepackage{amscd}
\usepackage{hyperref}
\hypersetup{colorlinks=true,citecolor=green,linkcolor=blue,urlcolor=blue}
\usepackage{tipa}
\usepackage{amssymb}
\usepackage{amsfonts}
\usepackage{stmaryrd}
\usepackage{amssymb}
\usepackage{titling}
\usepackage{CJK}  
\usepackage{indentfirst} 

\usepackage{latexsym}   
\usepackage{bm}         

\usepackage{pifont}
\usepackage{bm}
\usepackage{amsmath,amssymb,amsfonts}
\usepackage{indentfirst}

\usepackage{xcolor}
\usepackage{cases}
\usepackage{wasysym}
\usepackage{amsthm}

\newtheorem{theorem}{Theorem}[section]
\newtheorem{lemma}[theorem]{Lemma}
\newtheorem{proposition}[theorem]{Proposition}

\newtheorem{remark}[theorem]{Remark}
\newtheorem{corollary}[theorem]{Corollary}
\newtheorem{definition}[theorem]{Definition}

\CJKtilde   

\begin{document}

\title{\Large \textbf{Restriction of $p$-modular representations of $U(2, 1)$ to a Borel subgroup}}
\date{}
\author{\textbf{Peng Xu}}
\maketitle

\begin{abstract}
Let $G$ be the unramified unitary group $U(2, 1)(E/F)$ over a non-archimedean local field $F$ of odd residue characteristic $p$, and let $B$ be the standard Borel subgroup of $G$. In this paper, we study the problem of the restriction of irreducible smooth $\overline{\mathbf{F}}_p$-representations of $G$ to $B$, and we prove results which are analogous to that of Pa$\check{\text{s}}$k$\bar{\text{u}}$nas on $GL_2 (F)$ (\cite{Pas07}).
\end{abstract}

\section{Introduction}\label{sec: intro}

Let $G$ be the unitary group $U(2, 1)(E/F)$ defined over a non-archimedean local field $F$ of odd residue characteristic $p$, and let $B$ be the standard Borel subgroup of $G$. In this paper, we investigate the restriction of irreducible smooth $\overline{\mathbf{F}}_p$-representations of $G$ to $B$.

\smallskip

Our first main result is concerning principal series of $G$:
\begin{theorem}(Corollary \ref{restriction to borel for irre principal series}, \ref{restriction to borel for reducible principal series})\label{main result for non-ss: intro}
Let $\pi$ be a smooth representation of $G$. We have:

$(1)$.~ Let $\varepsilon$ be a character of $B$ such that $\varepsilon \neq \eta \circ \textnormal{det}$ for any character $\eta$ of $E^\times$. Then,
\begin{center}
$\textnormal{Hom}_G (\textnormal{ind}^G _B \varepsilon, \pi) \cong \textnormal{Hom}_B (\textnormal{ind}^G _B \varepsilon, \pi)$.
\end{center}

$(2)$.~For the trivial character of $B$, we have
\begin{center}
$\textnormal{Hom}_G (\textnormal{ind}^G _B 1, \pi) \cong \textnormal{Hom}_B (St, \pi) $.
\end{center}
Here $St$ is the Steinberg representation of $G$.
\end{theorem}

For a $p$-adic split connected reductive group, general results on restriction of principal series to a Borel subgroup have been obtained by Vign$\acute{\text{e}}$ras (\cite{V08}). Her approach can be modified to work for certain non-split groups of small ranks (\cite{Abd-2011}, \cite{Ly15a}). Our result above considers another aspect of this problem.

\smallskip
The work of Abe--Henniart--Herzig--Vign$\acute{\text{e}}$ras(\cite{AHHV17}) gives a classification of irreducible \emph{admissible} mod-$p$ representations of a $p$-adic reductive group in terms of \emph{admissible} supersingular representations. Roughly speaking, supersingular representations are the mod-$p$ analogue of supercuspidal representations. However, besides the group $GL_2 (\mathbf{Q}_p)$ (\cite{Breuil03}) and a few closely related cases, supersingular representations remain mysterious largely; indeed, such representations might not even be \emph{admissible} in general, as is shown in the work of Le (\cite{Le2019}). Note that Herzig--Kozio{\l}--Vign$\acute{\text{e}}$ras have proved the existence of admissible supersingular representation for any $p$-adic connected reductive group over $F$ of characteristic $0$ (\cite{HKV20}).

\smallskip
The following is our main result on supersingular representations of $G$:

\begin{theorem}\label{main result on ss in intro}
We have:

$(1)$. (Theorem \ref{main result on ss})\label{main result on ss: intro} ~Let $\pi$ be a supersingular representation of $G$. Then
\begin{center}
$\pi \mid_B$ is irreducible.
\end{center}

$(2)$.~(Theorem \ref{hom_G= hom_B for ss})\label{main result on ss II: intro}
Give two smooth representations $\pi$ and $\pi'$ of $G$. Suppose $\pi$ is supersingular. Then, we have
\begin{center}
$\textnormal{Hom}_G (\pi, \pi') \cong \textnormal{Hom} _B (\pi, \pi').$
\end{center}
\end{theorem}

\smallskip
An immediate application of $(1)$ above gives that the usual Jacquet module of a supersingular representation vanishes (Corollary \ref{vanishe of jacquet module}).
\medskip

 Our results are analogous to results of Pa$\check{\text{s}}$k$\bar{\text{u}}$nas on $GL_2 (F)$ (\cite{Pas07}), and we follow his strategy closely. To complete the proof of $(2)$ of Theorem \ref{main result on ss in intro}, we came to a new phenomenon which does not exist for $GL_2$. The operator $S_-$ is an analogue of the element $\Pi=\begin{pmatrix} 0  & 1  \\  \varpi_F & 0
\end{pmatrix}$ in $GL_2 (F)$, but in our case it can always happen that $S_- \cdot v =0 $ for some $v\in \pi^{I_{1, K}}$. This causes essential troubles when we study mod-$p$ representations of the group $G$. For the problem considered in this paper, we conquer such difficulty, see the argument of Theorem \ref{hom_G= hom_B for ss}. (For unexplained notations, see section \ref{sec: notations} and section \ref{sec: S_K and S_-}).

  \medskip
  When $F=\mathbf{Q}_p$, Pa$\check{\text{s}}$k$\bar{\text{u}}$nas' results were firstly discovered by Berger (\cite{Berger10}), where he uses the theory of $(\varphi, \Gamma)$-modules and classification of supersingular representations. In the work of Colmez on $p$-adic local Langlands correspondence of $GL_2 (\mathbf{Q}_p)$ (\cite{Col2010}), the restriction to a Borel subgroup plays a prominent role. We expect our results would also have some interesting arithmetic applications in the future.

\begin{remark}
Most part of this work (except for (2) of Theorem \ref{main result on ss II: intro}) was done when the author was a postdoc at Einstein Institute of Mathematics (2017-2018), and versions of that have been put on Arxiv in early of 2019 (see \cite{X2019a}). Sometime after that, we were aware of that Abdellatif and Hauseux have announced their results on the same problem in which they work for groups of semi-simple rank one (\cite{Abd-2021}). As far as we know (as of April/2024), their work has not appeared yet, and indeed some of their argument is close to ours.
\end{remark}

\smallskip
This paper is organized as follows. In section \ref{sec: notations}, after setting up general notations, we recall some preliminaries on weights and the Hecke operator $T$. In section \ref{sec: S_K and S_-}, we define certain $I_{1, K}$-invariant maps $S_K$ and $S_-$, and verify their basic properties. In section \ref{sec: non-supersingular}, we prove Theorem \ref{main result for non-ss: intro}. In section \ref{sec: supersingular}, we prove Theorem \ref{main result on ss: intro}.

\section{Notations and Preliminaries}\label{sec: notations}

\subsection{General notations}
Let $E/F$ be a unramified quadratic extension of non-archimedean local fields of odd residue characteristic $p$. Let $\mathfrak{o}_E$ be the ring of integers of $E$, $\mathfrak{p}_E$ be the maximal ideal of $\mathfrak{o}_E$, and $k_E = \mathfrak{o}_E /\mathfrak{p}_E$ be the residue field. Fix a uniformizer $\varpi_{E}$ in $E$. Equip $E^3$ with the Hermitian form h:
\begin{center}
 $\text{h}:~E^3 \times E^3 \rightarrow E$,
$(v_{1}, v_{2}) \mapsto ~v_{1}^{\text{T}}\beta \overline{v_2}, v_{1}, v_{2}\in E^3$.
\end{center}
Here, $-$ is a generator of $\text{Gal}(E/F)$, and
$\beta$ is the matrix
\[ \begin{matrix}\begin{pmatrix} 0  & 0 & 1  \\ 0  & 1 & 0\\
1 & 0 & 0
\end{pmatrix}
\end{matrix}. \]
The unitary group $G$ is defined as:

\begin{center}
$G=\{g\in \text{GL}(3, E)\mid \text{h}(gv_1, gv_2)= \text{h}(v_1, v_2), \forall~v_1, v_2~\in E^3\}.$
\end{center}

Let $B=HN$ (resp, $B'= HN'$) be the subgroup of upper (resp, lower) triangular matrices of $G$, with $N$ (resp, $N'$) the unipotent radical of $B$ (resp, $B'$) and $H$ the diagonal subgroup of $G$. A typical element in $H$ is of the following form and is denoted by $h(x, y)$:
\begin{center}
$h(x, y)= \begin{pmatrix}  x & 0 &  0 \\ 0 & y & 0\\
0 & 0 & \bar{x}^{-1}
\end{pmatrix}$
\end{center}
for $x\in E^\times, y \in E^1$. We will write $h(x, -\bar{x}x^{-1})$ as $h(x)$ for short. An element in $N$ and $N'$ is of the following form
\begin{center}
$\begin{pmatrix}  1 & x & y  \\ 0 & 1 & -\bar{x}\\
0 & 0 & 1
\end{pmatrix}$, ~
$\begin{pmatrix}  1 & 0 & 0   \\ x & 1 & 0\\
y & -\bar{x} & 1
\end{pmatrix}$
\end{center}
and are denoted by $n(x, y)$ and $n'(x, y)$. Here we recall that $(x, y)\in E^2$ satisfies the relation $x\bar{x}+ y+ \bar{y}=0$.

For any $k\in \mathbb{Z}$, denote by $N_k$ and $N'_k$ respectively the following subgroups of $N$ and $N'$
\begin{center}
$N_k = \{n(x, y)\in N \mid y \in \mathfrak{p}^k _E\}$,

$N'_k = \{n'(x, y)\in N' \mid y \in \mathfrak{p}^k _E\}$.
\end{center}

We record a useful identity in $G$: for $y\neq 0$,
\begin{equation}\label{useful identity}
\beta n(x, y)= n(\bar{y}^{-1}x, y^{-1})\cdot h(\bar{y}^{-1})\cdot n'(-\bar{y}^{-1}\bar{x}, y^{-1}).
\end{equation}

\medskip
Up to conjugacy, the group $G$ has two maximal compact open subgroups $K_0$ and $K_1$, given by:
\begin{center}
$K_0= \begin{pmatrix}  \mathfrak{o}_E & \mathfrak{o}_E & \mathfrak{o}_E  \\ \mathfrak{o}_E  & \mathfrak{o}_E & \mathfrak{o}_E\\
\mathfrak{o}_E & \mathfrak{o}_E & \mathfrak{o}_E
\end{pmatrix}\cap G, ~K_1= \begin{pmatrix}  \mathfrak{o}_E & \mathfrak{o}_E & \mathfrak{p}^{-1}_E  \\ \mathfrak{p}_E  & \mathfrak{o}_E & \mathfrak{o}_E\\
\mathfrak{p}_E & \mathfrak{p}_E & \mathfrak{o}_E
\end{pmatrix}\cap G.$
\end{center}
The maximal normal pro-$p$ subgroups of $K_0$ and $K_1$ are respectively:

$K^1 _0= 1+\varpi_E M_3 (\mathfrak{o}_E)\cap G, ~K^1 _1= \begin{pmatrix}  1+\mathfrak{p}_E & \mathfrak{o}_E & \mathfrak{o}_E  \\ \mathfrak{p}_E  & 1+\mathfrak{p}_E & \mathfrak{o}_E\\
\mathfrak{p}^2_E & \mathfrak{p}_E & 1+\mathfrak{p}_E
\end{pmatrix}\cap G.$

Let $\alpha$ be the following diagonal matrix in $G$:
\[ \begin{matrix}\begin{pmatrix} \varpi_{E}^{-1}  & 0 & 0  \\ 0  & 1 & 0\\
0 & 0 & \varpi_{E}
\end{pmatrix}
\end{matrix} ,\]
and put $\beta'=\beta \alpha^{-1}$. Note that $\beta\in K_0$ and $\beta'\in K_1$.  We use $\beta_K$ to denote the unique element in $K\cap \{\beta, \beta'\}$.

Let $K\in \{K_0, K_1\}$, and $K^1$ be the maximal normal pro-$p$ subgroup of $K$. We identify the finite group $\Gamma_K= K/K^1$ with the $k_F$-points of an algebraic group defined over $k_F$. Let $\mathbb{B}$ (resp, $\mathbb{B}'$) be the upper (resp, lower) triangular subgroup of $\Gamma_K$, and $\mathbb{U}$ (resp, $\mathbb{U}'$) be its unipotent radical. The Iwahori subgroup $I_K$ (resp, $I'_K$) and pro-$p$ Iwahori subgroup $I_{1,K}$ (resp, $I' _{1,K}$) in $K$ are the inverse images of $\mathbb{B}$ (resp, $\mathbb{B}'$) and $\mathbb{U}$ (resp, $\mathbb{U}'$).

\smallskip
Denote by $n_K$ and $m_K$ the unique integers such that $N\cap I_{1, K}= N_{n_K}$ and $N'\cap I_{1, K}=N'_{m_K}$. We have $n_K +m_K =1$. Note that the coset spaces $N_{n_K}/N_{n_K +1}$ and $N'_{m_K}/N'_{m_K +1}$ are indeed groups of order respectively $q^{t_K}$ and $q^{4-t_K}$, where $t_K= 3$ or $1$, depending on $K$ is hyperspecial or not.

\medskip

All representations in this note are smooth over $\overline{\mathbf{F}}_p$.

\subsection{Weights}\label{subsec: weights}
Let $\sigma$ be an irreducible smooth representation of $K$. As $K^1$ is pro-$p$ and normal in $K$, $\sigma$ factors through the finite group $\Gamma_K$, i.e., $\sigma$ is the inflation of an irreducible representation of $\Gamma_K$. Conversely, any irreducible representation of $\Gamma_K$ inflates to an irreducible smooth representation of $K$. We may therefore identify irreducible smooth representations of $K$ with irreducible representations of $\Gamma_K$, and we shall call them \emph{weights} of $K$ or $\Gamma_K$ from now on.

For a weight $\sigma$ of $K$, it is known that $\sigma^{I_{1,K}}$ and $\sigma_{I'_{1, K}}$ are one-dimensional, and that the natural composition map $\sigma^{I_{1,K}}\hookrightarrow \sigma \twoheadrightarrow \sigma_{I'_{1,K}}$ is an isomorphism of vector spaces (\cite[Theorem 6.12]{C-E2004}). This implies there exists a unique $\lambda_{\beta_K, \sigma}\in \overline{\mathbf{F}}_p$, such that $\beta_K \cdot v-\lambda_{\beta_K, \sigma}v\in \sigma (I'_{1,K})$, for $v\in \sigma^{I_{1,K}}$. By \cite[Proposition 3.16]{H-V2011} and the fact that $\beta_K \notin I_{K}\cdot I'_{K}$, the scalar $\lambda_{\beta_K, \sigma}$ is zero if $\text{dim} \sigma > 1$, and is equal to $\sigma (\beta_K)$ if $\text{dim} \sigma =1$.

\subsection{The Hecke operator $T$}\label{subsec: definition of T}
 Let $K \in \{K_0, K_1\}$, and $\sigma$ be a weight of $K$. Let $\text{ind}_K ^G \sigma$ be the maximal compact induction  and $\mathcal{H}(K, \sigma):=\text{End}_G (\text{ind}_K ^G \sigma)$ be the associate spherical Hecke algebra. The algebra $\mathcal{H}(K, \sigma)$ is isomorphic to $\overline{\mathbf{F}}_p [T]$, for certain $T \in \mathcal{H}(K, \sigma)$ (\cite[Corollary 1.3]{Her2011a}, see also \cite[Proposition 3.3]{X2019}).

We don't recall the exact definition of $T$ but only its formula on a specific function. For a non-zero vector $v\in \sigma$, denote by $\hat{f}_v$ the function in $\text{ind}^G _K \sigma$ supported on $K$ and having value $v$ at $Id$.
\begin{proposition}\label{T[Id, v_0]}
Take a non-zero vector $v_0$ in $\sigma^{I_{1, K}}$. Then, we have
\begin{equation}
T\hat{f}_{v_0}=\sum_{u\in N_{n_K}/N_{n_K+ 2}}~u\alpha^{-1}\cdot \hat{f}_{v_0}+ \lambda_{\beta_K, \sigma}\sum_{u\in N_{n_K +1}/N_{n_K+ 2}}\beta_K u \alpha^{-1}\cdot \hat{f}_{v_0}
\end{equation}
\end{proposition}

\begin{proof}
This is \cite[Proposition 3.6]{X2019}. Note that the above formula determines $T$ uniquely, as the function $\hat{f}_{v_0}$ generates the whole representation $\text{ind}^G _K \sigma$.
\end{proof}

\section{The $I_{1, K}$-invariant maps $S_K$ and $S_-$}\label{sec: S_K and S_-}

In this section, we recall some partial linear operators on a smooth representation $\pi$, and their certain invariant properties.

\begin{definition}
Let $\pi$ be a smooth representation of $G$. We define:
\begin{center}
$S_K: \pi^{N'_{m_K}} \rightarrow \pi^{N_{n_K}}$,

$v \mapsto \sum_{u \in N_{n_K}/N_{n_K +1}}~u \beta_K v$.
\end{center}

\begin{center}
$S_-: \pi^{N_{n_K}} \rightarrow \pi^{N'_{m_K}}$,

$v\mapsto \sum_{u\in N _{n_K +1 }/N _{n_K +2}}~\beta_K u\alpha^{-1}v$
\end{center}

\end{definition}

It is simple to check both $S_K$ and $S_-$ are well-defined.
\begin{proposition}\label{S_K and S_- preserve I_1}
We have:

$(1)$.~~Let $h\in H_0 =I_K \cap H$. Then $S_K (hv)= h^s\cdot S_K v$, for $v\in \pi^{N'_{m_K}}$, and $S_- (hv)= h^s\cdot S_- v$, for $v\in \pi^{N_{n_K}}$, where $h^s$ is short for $\beta_K h \beta_K$.

$(2)$.~~If $v$ is fixed by $I_{1, K}$, the same is true for $S_K \cdot v$ and $S_- \cdot v$.
\end{proposition}

\begin{proof}

For $(1)$, we note that the group $H_0$ acts on $\pi^{N_{n_K}}$ and $\pi^{N'_{m_K}}$, as it normalizes $N_{n_K}$ and $N'_{m_K}$. The statement then follows from the definitions.

To prove $(2)$, we need the following Lemma.
\begin{lemma}\label{exchange lemma}
Given a $u'\in N'_{m_K}$ and a $u\in N_{n_K}$.

$(1)$. There is a unique $u_1\in N_{n_K}, h\in H_1, u'_1 \in N'_{m_K}$ so that the following identity
\begin{center}
$u'u= u_1 h u'_1$
\end{center}
holds.

$(2)$. For any $l > m\geq 0$, when $u$ goes through $N_{n_K +m} / N_{n_K +l}$, the element $u_1$ also goes through $N_{n_K +m}/ N_{n_K +l}$.
\end{lemma}

\begin{proof}
The uniqueness statement is clear, and only the existence needs to be proved. Assume $u =n(x, y) \in N, u' \in n'(x', y') \in N'$. Then, if $1+x x' + \overline{y y'}\in E^\times$, we have
 \begin{center}
 $u'u= u_1 h u'_1$,
\end{center}
where $u_1= n(x_1, y_1) \in N$ in which $x_1, y_1$ are given by
\begin{center}
$x_1= \frac{x- \overline{y x'}}{1+x x'+\overline{y y'}}, y_1= \frac{y}{1+\overline{x x'}+y y'}$,
\end{center}
and $h$ is the diagonal matrix
\begin{center}
$\begin{pmatrix}  \frac{1}{1+x x' + \overline{y y'}} & 0 & 0  \\ 0 & \frac{1+x x'+ \overline{y y'}}{1+ \overline{x x'} +y y'} & 0\\
0 & 0 & 1+ \overline{x x'}+ y y'
\end{pmatrix},$
\end{center}
and $u'_1= n'(x'_1, y'_1) \in N'$ in which
\begin{center}
$x'_1 = \frac{x'- \overline{xy'}}{1+ xx' +\overline{yy'}}, y'_1 = \frac{y'}{1+ \overline{xx'}+yy'}$.
\end{center}

Under our assumption that $u'\in N'_{m_K}$ and $u\in N_{n_K}$, the condition $1+x x' + \overline{y y'}\in E^\times$ holds automatically. The existence is established.

\medskip
We continue to prove $(2)$. From the formulae of $x_1$ and $y_1$ given above, one checks by a direct computation that
\begin{center}
$u_1 \in u N_{n_K +m +1}$, if $u\in N_{n_K +m}$ for some $m\geq 0$.
\end{center}
Explicitly,
\begin{equation}\label{u_1= uN+}
u_1 =u\cdot n(\ast, yz)
\end{equation}
holds for some $z\in \mathfrak{p}_E$. Recall that $u=n(x, y)\in N_{n_K}, u'= n'(x', y')\in N'_{m_K}$.

We may therefore view $u'$ as a map
\begin{center}
$u': N_{n_K +m} / N_{n_K +l} \rightarrow N_{n_K +m} / N_{n_K +l}$

$u N_{n_K +l} \mapsto u_1 N_{n_K +l}$
\end{center}

It suffices to show the map is injective. Assume for another $w\in N_{n_K}$, we have a decomposition $u'w= u_2 b''$ for $u_2\in N_{n_K}$ and $b''\in B'$. We have to prove:
\begin{center}
 $u_2\in u_1 N_{n_K +l}$ implies $w\in u N_{n_K +l}$.
\end{center}
 Write $u^{-1}_1 u_2$ as $u_3$. A little algebraic transform gives:
\begin{center}
$w= u \cdot b'^{-1}u_3 b''$
\end{center}
We need to check that the element $b'^{-1}u_3 b'' \in N_{n_K}$, denoted by $u_4$, lies in $N_{n_K +l}$. The element $b'$ can be written as $h\cdot u'_1$, for a diagonal matrix $h\in H_1$ and $u'_1 \in N'_{m_K}$. We therefore get
\begin{center}
$u'_1 u_4= (h^{-1}u_3 h)\cdot h^{-1}b''$,
\end{center}
where the right hand side is a decomposition of $u'_1 u_4$ given in $(1)$. The uniqueness of such a decomposition implies
\begin{center}
$u_4\in N_{n_K +l}$ if and only if $h^{-1}u_3 h\in N_{n_K +l}$
\end{center}
for any $l\geq 0$. Our assumption is that $u_3= u^{-1}_1 u_2\in N_{n_K +l}$, which is the same as $h^{-1}u_3 h\in N_{n_K +l}$ ($h\in H_1$). We are done.
\end{proof}

We proceed to complete the argument of $(2)$ of the Proposition. By $(1)$ and the decomposition of $I_{1, K}= N'_{m_K} \times H_1 \times N_{n_K}$, it suffices to check that, for $u'= n'(x, y)\in N'_{m_K}$, the element $u'\cdot S_K v$
\begin{center}
$u'\cdot S_K v= \sum_{u\in N_{n_K}/N_{n_K +1}}~u'u\beta_K v$
\end{center}
is still equal to $S_K v=\sum_{u\in N_{n_K}/N_{n_K +1}}~u\beta_K v$. By $(1)$ of Lemma \ref{exchange lemma}, the right hand side of above sum is equal to:
\begin{center}
$\sum_{u\in N_{n_K}/N_{n_K +1}}~u_1 h u'_1 \beta_K v$.
\end{center}
 We get:
\begin{center}
$u'\cdot S_K v=\sum_{u\in N_{n_K}/N_{n_K +1}}~u_1\beta_K (\beta_K h u'_1 \beta_K) v =\sum_{u\in N_{n_K}/N_{n_K +1}}~u_1\beta_K v$,
\end{center}
which, by $(2)$ of Lemma \ref{exchange lemma}, is just $\sum_{u_1\in N_{n_K}/N_{n_K +1}}~u_1\beta_K v$. The argument for the statement $S_K v \in \pi^{I_{1, K}}$ for $v\in \pi^{I_{1, K}}$ is now complete.

\smallskip
By almost the same argument, one can verify that $S_-\cdot v \in \pi^{I_{1, K}}$ for $v\in \pi^{I_{1, K}}$.
\end{proof}

\begin{remark}\label{refinement of exchange lemme}
A slight variant of $(2)$ holds by the same argument. When $u$ goes through $(N_{n_K +m} \smallsetminus N_{n_K +n}) / N_{n_K +l}$, the element $u_1$ also goes through $(N_{n_K +m}\smallsetminus N_{n_K +n})/ N_{n_K +l}$, for any $l > n \geq m\geq 0$.
\end{remark}
\smallskip

We apply the operators $S_-$ and $S_K$ to the $I_{1, K}$-invariants of a principal series $\text{ind}^G _B \varepsilon$. The space $(\text{ind}^G _B \varepsilon)^{I_{1, K}}$ is two dimensional and a basis of that is given by the functions $g_1$ and $g_2$: the function $g_1$ is supported on $BI_K$ and satisfies $g_1 (Id) =1$; the function $g_2$ is supported on $B\beta_K I_K $ and satisfies that $g_2 (\beta_K)=1$.

\begin{proposition}\label{action of S_- and S_K on principal series}
We have:

$(1)$.~ $S_K g_1 =g_2, ~S_K g_2 = d_{q^{t_K}} \cdot g_2$.

$(2)$.~$S_- g_1 = d_{q^{4- t_K}} \cdot g_1, ~S_- g_2= \varepsilon (\alpha) g_1$.

\noindent Here, $d_{q^{t_K}}= \sum_{n(x, t)\in (N_{n_K} \smallsetminus N_{n_K +1})/ N_{n_K +1}} \varepsilon_0 ((h(t))$,

\noindent $d_{q^{4- t_K}}=\sum_{n(x, t)\in (N_{n_K +1} \smallsetminus N_{n_K +2})/ N_{n_K +2}} \varepsilon_0 ((h(t))$, $\varepsilon_0= \varepsilon\mid_{B\cap K}$.
\end{proposition}

\begin{proof}
By Proposition \ref{S_K and S_- preserve I_1}, it suffices to compute the values of the functions in consideration at $Id$ and $\beta_K$. We omit the details. Note that the exact values of $d_{q^{t_K}}$ and $d_{q^{4- t_K}}$ depend on the nature of $\varepsilon_0$, see \cite[Appendix A]{Karol-Peng2012}.
\end{proof}

Later on, we will use the composition $S_+ =  S_K \circ S_-$, which lies in $\text{End}_{\overline{\mathbf{F}}_p}(\pi^{N_{n_K}})$. Explicitly for $v \in \pi^{N _{n_K}}$, we have
\begin{center}
$S_+ v = \sum_{u\in N_{n_K} /N_{n_K +2}}~u\alpha^{-1}v$.
\end{center}
By Proposition \ref{S_K and S_- preserve I_1}, it preserves the $I_{1, K}$-invariants of a smooth representation.

\section{Non-supersingular representations}\label{sec: non-supersingular}

\subsection{Some recaps}\label{subsec: recap on principal series}
In this subsection, we recall briefly the restriction of a principal series to a Borel subgroup. For readers' convenience, we reproduce certain details below where we mainly follow the approach in \cite{V08}.

For a character $\varepsilon$ of $B$, consider the principal series $\text{ind}^G _B \varepsilon$. Recall that $\text{ind}^G _B \varepsilon$ is reducible if and only if $\varepsilon= \eta\circ \text{det}$ for some character $\eta$, and in this case it is of length two.
Evaluating an $f \in \text{ind}^G _B \varepsilon$ at the identity, we get a $B$-map from the principal series to the character $\varepsilon$. Denote the kernel by $\kappa_\varepsilon$. Then we have a short exact sequence of $B$-representations:
\begin{center}
$0\rightarrow \kappa_\varepsilon \rightarrow \text{ind}^G _B \varepsilon \rightarrow \varepsilon \rightarrow 0$
\end{center}

By almost the same argument of \cite[Theorem 5]{V08}, one shows that the $B$-representation $\kappa_\varepsilon$ is irreducible. Indeed, as explained below, one may even prove $\kappa_\varepsilon$ is irreducible as a representation of $\alpha^{\mathbb{Z}}N$. But by the same map $\Phi$ below, one may verify that $St\mid_B \cong \kappa_1$. This gives irreducibility of $St\mid_B$.

\begin{lemma}\label{irr of kappa}
The $B$-representation $\kappa_\varepsilon$ is irreducible.
\end{lemma}

\begin{proof}
1).~ We firstly identify the underlying space of $\kappa_\varepsilon$ with $C^\infty _c (N)$.

$\Phi: \kappa_\varepsilon \rightarrow C^\infty _c (N), f \mapsto \Phi (f), \Phi(f)(u) = f(\beta u), \forall u\in N.$

$\Psi: C^\infty _c (N) \rightarrow \kappa_\varepsilon, f\mapsto \Psi (f), \Psi (f) (b\beta u) = \varepsilon(b)f(u), \Psi (f)(b)=0, \forall b\in B, u \in N.$

For $b=hu_1\in B$ where $h\in H, u_1\in N$, and $f\in C^\infty _c (N)$, we put $b\cdot f (u)=\varepsilon (h^s) f((h^{-1} u h)  u_1)$, where $h^s$ denotes $\beta h \beta$. This gives $C^\infty _c (N)$ a structure of $B$-representation. We check easily that $\Psi$ and $\Phi$ are both $B$-equivariant, and are inverse to each other.

2).~We modify the argument of \cite[Proposition 5.2]{Ly15a} to our case. Let $V$ be a non-zero $B$-stable subspace of $C^\infty _c (N)$, and $f$ be a non-zero function in $V$. As $f$ is compactly supported and $N$ has the decreasing open compact cover $(N_k)_{k\in \mathbb{Z}}$, we may assume the support of $f$ is contained in $N_k$ for some integer $k$. Write $V_k$ for the subspace of $V$ consisting of functions supported in $N_k$, we have $V_k \neq 0$. By \cite[Lemma 1]{B-L95}, we know $V^{N_k} _k \neq 0$. This shows that $V$ contains the characteristic function $1_{N_k}$ of $N_k$.

Now for any $n \in \mathbb{Z}, u\in N$, we have $u\alpha^{n} \cdot 1_{N_k} =\varepsilon(\alpha^{-n}) 1_{N_{k-2n} u^{-1}}$, and as $V$ is $B$-stable we conclude $V$ contains $1_{N_{k-2n} u^{-1}}$. Note that $1_{N_{k-1}}= \sum_{u\in N_{k-1}/N_k}  u^{-1} \cdot 1_{N_k}$, so we have $1_{N_{k-1}} \in V$. Repeating the previous process, we conclude $V$ contains $1_{N_{k-2n-1} u^{-1}}$. In all we have shown $V$ contains all the functions $1_{N_k u}$ for any $k\in \mathbb{Z}$ and $u \in N$. As all the functions $\{1_{N_k u} \mid k \in \mathbb{Z}, u\in N\}$ span the underlying space of $C^\infty _c (N)$, we get $V= C^\infty _c (N)$.
\end{proof}

\subsection{Proof of Theorem \ref{main result for non-ss: intro}}
We now come to the main input of this section.

\begin{theorem}\label{main result for non-supersingular}
Let $\pi$ be any smooth representation of $G$. The restriction map induces an isomorphism between the following spaces:
\begin{center}
$\textnormal{Hom}_G (\textnormal{ind}^G _B \varepsilon, \pi) \cong \textnormal{Hom}_B (\kappa_\varepsilon, \pi)$
\end{center}
\end{theorem}

\begin{proof}
We show firstly that the restriction map is injective.

Given $\phi_1$ and $\phi_2$ in the space $\textnormal{Hom}_G (\textnormal{ind}^G _B \varepsilon, \pi)$, suppose that $\phi= \phi_1 - \phi_2$ vanishes at the subspace $\kappa_\varepsilon$. By the remark proceeding Lemma \ref{irr of kappa}, $\phi$ induces a $B$-map from the character $\varepsilon$ to $\pi$, for which we still denote by $\phi$.

\begin{lemma}\label{space of B-maps from a character}
If $\varepsilon \neq \eta\circ \textnormal{det}$ for any character $\eta$, then $\textnormal{Hom}_B (\varepsilon, \pi) =0$.
\end{lemma}
\begin{proof}
Assume $\phi \neq 0$. As $\pi$ is smooth, the vector $\phi (1)\in \pi$ is fixed by some $N'_{m_K +2k}$ for large enough $k$. Using the following identity repeatedly
\begin{center}
$\alpha N'_{m_K +2k-2} \alpha^{-1} = N'_{m_K +2k}$
\end{center}
and $\phi (\alpha \cdot 1)= \varepsilon(\alpha)\phi (1)= \alpha\cdot \phi(1)$, we see $\phi(1)$ is fixed by $N'$. As the group $G$ is generated by $B$ and $N'$, we see $\varepsilon$ extends uniquely to a character of $G$ (put $\varepsilon (N')=1$). In such a situation, $\textnormal{Hom}_B (\varepsilon, \pi) \cong\textnormal{Hom}_G (\varepsilon, \pi)$.
\end{proof}

\begin{remark}\label{remarks on regular character}
Under the same assumption on $\varepsilon$, the Lemma implies that any non-zero map in $\textnormal{Hom}_B (\textnormal{ind}^G _B \varepsilon, \pi)$ is an injection. Take $\pi= \textnormal{ind}^G _B \varepsilon$. We conclude that $\textnormal{End}_B (\textnormal{ind}^G _B \varepsilon)$ is one-dimensional. This is because, by the proceeding remark, any non-zero map in the former space will induce a non-zero map in $\textnormal{End}_B (\kappa_\varepsilon)$ which is one-dimensional by Lemma \ref{irr of kappa}. We deduce that $\textnormal{End}_B (\textnormal{ind}^G _B \varepsilon) \cong \textnormal{End}_G (\textnormal{ind}^G _B \varepsilon)$.
\end{remark}

We are done if $\varepsilon \neq \eta\circ \text{det}$ for any character $\eta$. Assume $\varepsilon =\eta\circ \text{det}$ for some $\eta$. After a twist we may assume $\eta=1$. If $\phi\neq 0$, it induces a non-zero map in $\textnormal{Hom}_B (1, \pi)$, which by the argument of Lemma \ref{space of B-maps from a character} is in $\textnormal{Hom}_G (1, \pi)$. This implies the map $\phi \in \textnormal{Hom}_G (\textnormal{ind}^G _B 1, \pi)$ realizes the trivial character of $G$ as a quotient of $\textnormal{ind}^G _B 1$, which is not true.

\smallskip

We proceed to prove the restriction map is surjective.

Recall again that the space $(\textnormal{ind}^G _B \varepsilon)^{I_{1, K}}$ is two dimensional with a basis of functions $g_1$ and $g_2$ characterized by: $g_1 (Id)=1, g_1 (\beta_K)=0, g_2 (Id)=0, g_2 (\beta_K)=1$. By Proposition \ref{action of S_- and S_K on principal series}, we have
\begin{equation}\label{S_+ g_2}
S_+ g_2 = \varepsilon (\alpha) g_2
\end{equation}
Then, by Lemma \ref{the K-sub generated by S+ v} the $K$-representation $\langle K\cdot g_2 \rangle= \langle K\cdot S_+ g_2 \rangle$ is a weight, denoted by $\sigma$, of dimension greater than one (note that $I_K$ acts on $g_2$ by the character $\varepsilon^s _0$).

Let $\phi$ be a non-zero $B$-map from $\kappa_\varepsilon$ to $\pi$. The function $g_2$ by definition is supported on $B\beta_K I_K$ so it lies in $\kappa_\varepsilon$. As $\kappa_\varepsilon$ is irreducible (Lemma \ref{irr of kappa}), we have $\phi (g_2)$ is non-zero. Since $\phi$ respects the action of $B$, the vector $\phi (g_2)$ is fixed by $B \cap I_{1, K}$. Now we compute $\phi (S_+ g_2)$:
\begin{center}
$\phi (S_+ g_2)=  \varepsilon(\alpha)\phi (g_2)= S_+ \phi (g_2) $
\end{center}
that is
\begin{equation}\label{eigenvalue of phi(g_2)}
\phi (g_2)= \varepsilon(\alpha)^{-1}S_+ \phi (g_2)
\end{equation}
As $\pi$ is smooth, the vector $\phi(g_2)$ is fixed by some $N'_{m_K +2k}$ for $k$ large enough. Now by applying the argument of Lemma \ref{exchange lemma} to the above equality, we see $\phi(g_2)$ is fixed by $N'_{m_K +2k-2}$. Repeating such a process enough times, we prove that the vector $\phi (g_2)$ is fixed by $N'_{m_K}$.
By Iwahori decomposition $I_{1, K}= N'_{m_K}\cdot(B \cap I_{1, K})$, we conclude that $\phi (g_2)$ is fixed by $I_{1, K}$.

By Lemma \ref{the K-sub generated by S+ v} again the representation $\langle K\cdot \phi (g_2)\rangle$ is a weight $\sigma'$ of dimension greater than one. \emph{We claim that $\sigma \cong \sigma'$}. The Iwahori group $I_K$ acts on the vector $\phi (g_2)$ by $\varepsilon^s _0$. If $\varepsilon_0\neq \eta\circ \text{det}$ for $K= K_0$ (resp, $\varepsilon_0\neq \varepsilon^s _0$ for $K= K_1$), the claim follows as in this case a weight is determined by the character of $I_K$ acting on its $I_{1, K}$-invariants. If it is in the other case, we are also done: neither $\sigma$ or $\sigma'$ is a one-dimensional character, and as quotients of the principal series $\text{Ind}^K _{I_K} \varepsilon^s _0$ they are both isomorphic to $st\otimes\varepsilon_0$.

By \cite[Proposition 4.15, 4.16]{X2018b}, we have an isomorphism, unique up to a scalar,
\begin{center}
$\text{ind}^G _K \sigma/ (T_\sigma - \varepsilon (\alpha)) \cong \text{ind}^G _B \varepsilon$.
\end{center}
As the representation $\langle G\cdot \phi(g_2)\rangle$ contains the weight $\sigma$, it is therefore a quotient of the above representation (by \eqref{eigenvalue of phi(g_2)} and that $\sigma$ is not a character). We have shown $\phi$ extends to a $G$-map as required.
\end{proof}

With the last Theorem proved, we conclude with the following corollaries.

\begin{corollary}\label{restriction to borel for irre principal series}
Suppose $\varepsilon\neq \eta\circ \textnormal{det}$ for any character $\eta$. Then
\begin{center}
$\textnormal{Hom}_G (\textnormal{ind}^G _B \varepsilon, \pi) \cong \textnormal{Hom}_B (\textnormal{ind}^G _B \varepsilon, \pi)$
\end{center}
\end{corollary}

\begin{proof}
Let $\phi \in \textnormal{Hom}_B (\textnormal{ind}^G _B \varepsilon, \pi)$ be non-zero. By Remark \ref{remarks on regular character}, we know $\phi$ is injective. We have:

1).~ Using Remark \ref{remarks on regular character} again, the image of $\phi$ is contained in $\langle G\cdot \phi(\kappa_\varepsilon) \rangle$.

2).~ Applying Theorem \ref{main result for non-supersingular} to $\pi'= \langle G\cdot \phi(\kappa_\varepsilon) \rangle$ and using irreducibility of $\textnormal{ind}^G _B \varepsilon$, there is an isomorphism $\text{ind}^G _B \varepsilon \cong \langle G\cdot \phi(\kappa_\varepsilon) \rangle$.

By 1), the map $\phi$ lies in $\textnormal{Hom}_B (\textnormal{ind}^G _B \varepsilon, \langle G\cdot \phi(\kappa_\varepsilon) \rangle)$ but by 2) the latter space is isomorphic to $\textnormal{End}_B (\textnormal{ind}^G _B \varepsilon)$. We deduce that $\phi$ is $G$-equivariant by the last assertion of Remark \ref{remarks on regular character}.
\end{proof}

\begin{corollary}\label{restriction to borel for reducible principal series}
We have
\begin{center}
$\textnormal{Hom}_G (\textnormal{ind}^G _B 1, \pi) \cong \textnormal{Hom}_B (St, \pi)$
\end{center}
\end{corollary}
\begin{proof}
As $St\mid_B \cong \kappa_1$ (remarks before Lemma \ref{irr of kappa}), the assertion is a special case of Theorem \ref{main result for non-supersingular}. Note that the result can not be improved by replacing $\textnormal{ind}^G _B 1$ in the statement by $St$: the space $\textnormal{Hom}_G (\textnormal{ind}^G _B 1, \textnormal{ind}^G _B 1)\neq 0$ but $\textnormal{Hom}_G (St, \textnormal{ind}^G _B 1)=0$.
\end{proof}

\section{Supersingular representations}\label{sec: supersingular}

\subsection{Definition}

Recall we have defined the Hecke operator $T$ in subsection \ref{subsec: definition of T}. To define the supersingular representations, we modify it in the following way. If $\textnormal{dim}~\sigma= 1$, we put $T_\sigma= T +1$; otherwise, we put $T_\sigma= T$. Note that this modification reflects the fact the group $G$ has two maximal compact open subgroups, up to conjugacy.

\begin{definition}\label{super reps}

An irreducible smooth $\overline{\mathbf{F}}_p$-representation $\pi$ of $G$ is called supersingular if it is a quotient of $\textnormal{ind}^G _K \sigma/(T_\sigma)$, for some weight $\sigma$ of $K$.
\end{definition}

\subsection{A key property}
Let $\pi$ be an irreducible smooth representation of $G$, and $\sigma$ be a weight of $K$ contained in $\pi$. By \cite[Theorem 1.1]{X2018a}, the representation $\pi$ admits Hecke eigenvalues for the spherical Hecke algebra $\mathcal{H}(K, \sigma)\cong \overline{\mathbf{F}}_p [T_\sigma]$. This implies that the representation $\pi$ is a quotient of $\text{ind}^G _K \sigma/ (T_\sigma -\lambda)$, for some scalar $\lambda$. By definition \ref{super reps}, the representation $\pi$ is supersingular if $\lambda=0$.

\begin{lemma}\label{nipotent}
Let $\pi$ be a supersingular representation of $G$, and assume $\phi$ is a non-zero $G$-map from $\textnormal{ind}^G _K \sigma$ to $\pi$. Then, for large enough $k\geq 1$, we have
\begin{center}
$\phi\circ T_{\sigma}^k =0.$
\end{center}
\end{lemma}
\begin{proof}
By \cite[Corollary 4.2]{X2018a}, there is a non-constant polynomial $P(X)$ such that $\phi\circ P(T_\sigma)=0$. Assume $P(X)$ is such a polynomial of minimal degree. Take a root $\lambda$ of $P(X)$, and write $P(X)= (X- \lambda) P_1 (X)$. Put $\phi'= \phi \circ P_1 (T_\sigma)$. Note that $\phi'$ is still a $G$-map from $\textnormal{ind}^G _K \sigma$ to $\pi$. By our assumption, the map $\phi'$ is non-zero and factors through $\textnormal{ind}^G _K \sigma / (T_\sigma -\lambda)$. As $\pi$ is supersingular, we have $\lambda =0$ (\cite[Theorem 1.1]{X2018b}). We conclude $P(X)= X^n$ for some $n\geq 1$.
\end{proof}

\smallskip
\begin{lemma}\label{the K-sub generated by S+ v}
Let $\pi$ be a smooth representation of $G$. Assume $v$ is a non-zero vector in  $\pi^{I_{1, K}}$, such that $I_K$ acts on $v$ as a character. Then,  either $S_+ v=0$, or $S_+ v$ generates a weight of $K$ of dimension greater than one.
\end{lemma}

\begin{proof}
Assume $S_+ v\neq 0$. We put $w= S_- v$. By definition,
\begin{center}
$S_+ v= S_K w$,
\end{center}
and we see $w$ must be non-zero. Consider the $K$-representation $\kappa= \langle K\cdot w\rangle$. As $I_K$ acts on $v$ by a character $\chi$, $I_K$ acts on $w$ by $\chi^s$ (Proposition \ref{S_K and S_- preserve I_1}). By Frobenius reciprocity, there is a surjective $K$-map from $\text{Ind}^K _{I_K} \chi^s$ to $\kappa$, sending $\varphi_{\chi^s}$ to $w$. Here, $\varphi_{\chi^s}$ is the function in $\text{Ind}^K _{I_K} \chi^s$ supported on $I_K$ and having value $1$ at $Id$.

Via aforementioned map, we see $\langle K\cdot S_+ v \rangle$ is the image of $\langle K \cdot S_K \varphi_{\chi^s} \rangle $.  But the latter, by \cite[Proposition 5.7]{Karol-Peng2012}, is an irreducible smooth representation of $K$ of dimension greater than one. The assertion follows.
\end{proof}

\begin{proposition}\label{nilpotent of S+}
 Assume $\pi$ is a supersingular representation of $G$, and $v$ is a non-zero vector in $\pi^{I_{1, K}}$.  Then, for $k\gg 0$, we have $S^k _+ v= 0$.
\end{proposition}

\begin{proof}
Assume firstly $I_K$ acts on $v$ as a character $\chi$.

Assume $S_+ v\neq 0$. By Lemma \ref{the K-sub generated by S+ v} the $K$-subrepresentation generated by $S_+ v$ is a weight of dimension greater than one, and denote it by $\sigma$. By Frobenius reciprocity, we have a $G$-map $\phi$ from $\textnormal{ind}^G _K \sigma$ to $\pi$, sending the function $\hat{f}_{S_+ v}$ to $S_+ v$. From Lemma \ref{nipotent}, there is some $k\geq 1$ such that
\begin{center}
$S_+ ^k v=0,$
\end{center}
and we are done in this special case.

Note that $I_K/I_{1,K}$ is an abelian group of finite order prime to $p$. For any non-zero $v\in \pi^{I_{1,K}}$, the $I_K$-representation $\langle I_K\cdot v \rangle$ generated by $v$ is a finite sum of characters, and we may write $v$ as $\sum v_i$ so that $I_K $ acts on $v_i$ by a character $\chi_i$ of $I_K/I_{1, K}$. We then apply the previous process to each $v_i$, and take the largest $k_i$. We are done.
\end{proof}

\subsection{A criteria of Pa$\check{\text{s}}$k$\bar{\text{u}}$nas}

In this part, for an irreducible smooth representation we prove a sufficient condition under which its restriction to the Borel subgroup remains irreducible. We will verify it for supersingular ones in the next part.

\begin{proposition}\label{restriction to borel}
Let $\pi$ be an irreducible smooth representation of $G$. If, for any non-zero vector $w\in \pi$, there is a non-zero vector $v\in \pi^{I_{1, K}}\cap \langle B\cdot w\rangle$ such that
\begin{center}
 $S_+ v=0$,
\end{center}
then $\pi\mid_B$ is irreducible.
\end{proposition}

\begin{proof}
Let $w$ be a non-zero vector in $\pi$. As $\pi$ is smooth, there exists a $k\geq 0$ such that $w$ is fixed by $N'_{2k+m_K}$. Hence, the vector $w_1=\alpha^{-k} w$ is fixed by $N'_{m_K}$. Since $I_{1, K}= (I_{1, K}\cap B)\cdot N'_{m_K}$, we see
\begin{center}
$\langle I_{1, K}\cdot w_1\rangle= \langle(I_{1, K}\cap B)\cdot w_1\rangle.$
\end{center}

As $I_{1, K}$ is pro-$p$, the space $\langle(I_{1, K}\cap B)\cdot w_1\rangle$ has non-zero $I_{1, K}$-invariant (\cite[Lemma 1]{B-L95}). We conclude that $\pi^{I_{1, K}}\cap \langle B\cdot w\rangle\neq 0$.

\begin{lemma}
If $S_+ v=0$, then $\beta_K v\in \langle B\cdot v\rangle.$
\end{lemma}

\begin{proof}
By the assumption $S_+ v=0$, we get
\begin{center}
$v= -\alpha\cdot \sum_{u\in (N_{n_K}\smallsetminus N_{n_K +2})/N_{n_K +2}}u\alpha^{-1} v$
\end{center}
or equivalently
\begin{equation}\label{beta_K v in  Bv}
\beta_K v= -\sum_{u\in (N_{n_K}\smallsetminus N_{n_K +2})/N_{n_K +2}}\beta_K \alpha u \alpha^{-1}v.
\end{equation}
Applying \eqref{useful identity}, we see $\beta_K \alpha u \alpha^{-1} \in BN'_{m_K}$, for any $u\in (N_{n_K}\smallsetminus N_{n_K +2})/N_{n_K +2}$. More precisely, for $u = n(\ast, \varpi^{n_K} _E t)\in N_{n_K}\smallsetminus N_{n_K +1}$, we have
\begin{center}
$\beta_K \alpha u \alpha^{-1}= n(\ast, \varpi^{n_K +2} _E t^{-1})h(\bar{t}^{-1})\alpha^{-2}n'(\ast, \varpi^{m_K +1} _E t^{-1})$;
\end{center}
for $u= n(\ast, \varpi^{n_K +1} _E t) \in N_{n_K +1}\smallsetminus N_{n_K +2}$, we have
\begin{center}
$\beta_K \alpha u \alpha^{-1}= n(\ast, \varpi^{n_K +1} _E t^{-1})h(\bar{t}^{-1})\alpha^{-1}n'(\ast, \varpi^{m_K} _E t^{-1}).$
\end{center}
That gives $\beta_K \alpha u \alpha^{-1}v \in  \langle B\cdot v \rangle$ for all $u\in (N_{n_K}\smallsetminus N_{n_K +2})/N_{n_K +2}$. We conclude $\beta_K v \in \langle B\cdot v\rangle$ from the above equality \eqref{beta_K v in  Bv}. For our later purpose, we record \eqref{beta_K v in  Bv} in a more explicit form:
\begin{equation}\label{beta_K v in  Bv: explicit form}
\beta_K v= -\sum_u u \alpha^{-2} h(\bar{t}^{-1})v -\sum_u u \alpha^{-1} h(\bar{t}^{-1})v,
\end{equation}
where, in the first sum $u= n(\ast, \varpi^{n_K +2} _E t)$ goes through $(N_{n_K +2}\smallsetminus N_{n_K +3})/N_{n_K +4}$; in the second sum $u= n(\ast, \varpi^{n_K +1} _E t)$ goes through $(N_{n_K +1}\smallsetminus N_{n_K +2})/N_{n_K +2}$.
\end{proof}

We proceed to complete the proof of Proposition \ref{restriction to borel}. Choose $0\neq v\in \pi^{I_{1, K}}\cap \langle B\cdot w\rangle$ such that $S_+ v=0$. The above Lemma says $\beta_K v\in \langle B\cdot v\rangle$. As $\pi$ is irreducible, we have $\pi=\langle G\cdot v\rangle$. By the Bruhat decomposition $G= BI_{1, K} \cup B\beta_K I_{1, K}$, we see
\begin{center}
$\pi\subseteq\langle B\cdot v\rangle \subseteq \langle B\cdot w\rangle$.
\end{center}
Hence, we have proved $\pi= \langle B\cdot w\rangle$ for any $w\in \pi$, and the proposition follows.
\end{proof}

\begin{remark}
The condition in Proposition \ref{restriction to borel} is merely sufficient: for the Steinberg representation $St$, its restriction to $B$ is irreducible, but $S_+$ does not annihilate the line $St^{I_{1, K}}$. More precisely, the representation $St$ is defined as $\textnormal{ind}^G _B 1 /(1)$, and its $I_{1, K}$-invariants is of one dimension. In the notation of section \ref{sec: S_K and S_-}, the space $St^{I_{1, K}}$ is spanned by the image of the function $g_1$. By Proposition \ref{action of S_- and S_K on principal series}, we check that
\begin{center}
$S_+ g_1 = -g_2$
\end{center}
which is just $S_+ \overline{g_1} =\overline{g_1}$.
\end{remark}

\subsection{Proof of $(1)$ of Theorem \ref{main result on ss: intro}}

\smallskip

\begin{theorem}\label{main result on ss}
Let $\pi$ be a supersingular representation of $G$. Then $\pi\mid_B$ is irreducible.
\end{theorem}

\begin{proof}
Let $w$ be any non-zero vector in $\pi$. We already know that (from the argument of Proposition \ref{restriction to borel})
\begin{center}
$\pi^{I_{1, K}} \cap \langle B\cdot w\rangle \neq 0$.
\end{center}
Take any non-zero vector $v$ in the above space. As $\pi$ is supersingular, by Proposition \ref{nilpotent of S+} $v$ will be annihilated by $S^k _+$ for $k$ large enough. Let $m$ be the least positive integer satisfying that. Now the vector $v'= S^{m-1} _+ v$ is non-zero. By Proposition \ref{S_K and S_- preserve I_1}, $v'$ is still $I_{1, K}$-invariant, and lies in $\langle B\cdot w\rangle$ by the form of $S_+$. It satisfies
\begin{center}
$S_+ v' =0$.
\end{center}
We are done by Proposition \ref{restriction to borel}.
\end{proof}

An immediate but interesting application of Theorem \ref{main result on ss} is the following. Recall the Levi decomposition $B= H \ltimes N$, where $H$ is the diagonal subgroup and $N$ is the upper unipotent radical.
\begin{corollary}\label{vanishe of jacquet module}
For a supersingular representation $\pi$ of $G$, we have
\begin{center}
$\pi_N =0$,
\end{center}
i.e., the usual Jacquet module of $\pi$ (with respect to $B$) vanishes.
\end{corollary}
\begin{proof}
Recall that $\pi_N = \pi / \pi (N)$, and the space $\pi (N) := \langle \pi(u)\cdot v -v \mid u\in N, v\in \pi \rangle$ is $B$-stable and non-zero. The assertion follows by Theorem \ref{main result on ss}.
\end{proof}

\subsection{Proof of $(2)$ of Theorem \ref{main result on ss: intro}}

\begin{theorem}\label{hom_G= hom_B for ss}
Give two smooth representations $\pi$ and $\pi'$ of $G$. Suppose $\pi$ is supersingular. Then, we have
\begin{center}
$\textnormal{Hom}_G (\pi, \pi') \cong \textnormal{Hom} _B (\pi, \pi').$
\end{center}
\end{theorem}

\begin{proof}
As we have one side inclusion $\textnormal{Hom}_G (\pi, \pi') \hookrightarrow \textnormal{Hom} _B (\pi, \pi')$, it suffices to prove any non-zero map $\phi$ in $\textnormal{Hom} _B (\pi, \pi')$ is also a $G$-map.

We begin with a non-zero vector $v \in \pi^{I_{1, K}}$, and by replacing $v$ by some other vector in the space $\langle I_K \cdot v\rangle$, we may assume $I_K$ acts on $v$ by a character $\chi$. As $\pi \mid_B$ is irreducible (Theorem \ref{restriction to borel}), the map $\phi$ is injective whence $\phi (v) \neq 0$. As $\pi'$ is smooth, $\phi (v)$ is fixed by $N'_{m_K + 2m}$ for some $m\geq 0$.

We assume $m\geq 1$, and proceed to find a non-zero vector $w \in \pi^{I_{1, K}} \cap \langle B\cdot v\rangle$ such that $\phi (w)$ is fixed by $N'_{m_K +2m-2}$.

\textbf{\emph{Case I: $S_+ v \neq 0$.}} We take $w$ as $S_+ v$, so it lies in $\pi^{I_{1, K}} \cap \langle B\cdot v\rangle$. Note that $I_K$ acts on $w$ still by the character $\chi$. Using Lemma \ref{exchange lemma}, we see the vector $\phi (w)= S_+ \phi (v)$ is fixed by $N'_{m_K +2m-2}$.

\textbf{\emph{Case II: $S_+ v = 0$ and $S_- v \neq 0$.}} We take $S_- v$ as $w$, which lies in $\pi^{I_{1, K}}$ (Lemma \ref{S_K and S_- preserve I_1}). Our assumption $S_+ v= 0$ means
\begin{center}
$S_K w =0$,
\end{center}
which is equivalent to
\begin{center}
$w= - \sum_{u\in (N_{n_K} \smallsetminus N_{n_K +1}) /N_{n_K +2}} \beta_K u \alpha^{-1} v$
\end{center}
An application of \eqref{useful identity} gives that for $u= n(\ast, \varpi^{n_K}_E t) \in N_{n_K} \smallsetminus N_{n_K +1}$
\begin{center}
$\beta_K u \alpha^{-1}= n(\ast, \varpi^{n_K}_E t^{-1})\alpha^{-1} h(\bar{t}^{-1})n'(\ast, \varpi^{m_K +1}_E t^{-1})$,
\end{center}
hence we have $\beta_K u \alpha^{-1} \cdot v = \chi(h(\bar{t}^{-1})) n(\ast, \varpi^{n_K}_E t^{-1})\alpha^{-1} \cdot v$. In all we get
\begin{center}
$w= -\sum_{u \in (N_{n_K} \smallsetminus N_{n_K +1}) /N_{n_K +2}} \chi(h(\bar{t}^{-1}))u\alpha^{-1} \cdot v \in \langle B\cdot v\rangle$,
\end{center}
whence that
\begin{center}
$\phi(w)= -\sum_{u \in (N_{n_K} \smallsetminus N_{n_K +1}) /N_{n_K +2}} \chi(h(\bar{t}^{-1}))u\alpha^{-1} \cdot \phi(v)$.
\end{center}
This, combined with Lemma \ref{exchange lemma} (especially \eqref{u_1= uN+}) and Remark \ref{refinement of exchange lemme}, shows $\phi (w)$ is fixed by $N'_{m_K +2m-2}$.

\textbf{\emph{Case III}}: $S_- v =0$ (whence $S_+ v=0$). By the definition of $S_-$, this assumption is same as $\sum_{u\in N_{n_K +1} /N_{n_K +2}} u\alpha^{-1} \cdot v =0$. Then by almost the same argument we have
\begin{center}
$\beta_K v= -\sum_{u\in (N_{n_K +1} \smallsetminus N_{n_K +2}) /N_{n_K +2}} \chi(h(\bar{t}^{-1}))u\alpha^{-1} \cdot v \in \langle B\cdot v\rangle$.
\end{center}
Thus
\begin{center}
$\phi(\beta_K v)= -\sum_{u\in (N_{n_K +1} \smallsetminus N_{n_K +2}) /N_{n_K +2}} \chi(h(\bar{t}^{-1}))u\alpha^{-1} \cdot \phi(v)$.
\end{center}
We conclude that $\phi (\beta_K v)$ is fixed by $N'_{m_K +2m-2}$, using Lemma \ref{exchange lemma} (especially \eqref{u_1= uN+}) and Remark \ref{refinement of exchange lemme}. Based on this, one verifies further that, for any $u \in N_{n_K} / N_{n_K +1}$, the vector $\phi (u\beta_K v)$ is still fixed by $N'_{m_K +2m-2}$ , using Lemma \ref{exchange lemma} again.

 $\bullet ~ S_K v =\sum_{u \in N_{n_K} / N_{n_K +1}} u\beta_K v\neq 0$. In this case, we take $S_K v$ as $w$. As
\begin{center}
$\phi (w)=\sum_{u \in N_{n_K} / N_{n_K +1}} \phi (u\beta_K v)$,
\end{center}
we conclude that $\phi (w)$ is fixed by $N'_{m_K +2m-2}$ by the proceeding remarks.

 $\bullet ~S_K v= \sum_{u \in N_{n_K} / N_{n_K +1}} u\beta_K v =0$. This assumption gives us that
\begin{center}
$v\in \langle u\beta_K v | u \in (N_{n_K} \smallsetminus N_{n_K +1}) /N_{n_K +1}\rangle$.
\end{center}
Explicitly, under this assumption we have
\begin{center}
$v= -\sum_{u\in (N_{n_K}\smallsetminus N_{n_K +1}) /N_{n_K +1}} \chi (h(\bar{t}^{-1}))u\beta_K \cdot v$
\end{center}
 Consider the $K$-representation $\tau= \langle K\cdot v \rangle$. It is spanned by the set $\{v, u\beta_K v \mid u \in N_{n_K} / N_{n_K +1}\}$. Then we see it can be spanned by the set
 \begin{center}
 $\{u\beta_K v \mid u \in (N_{n_K} \smallsetminus N_{n_K +1}) /N_{n_K +1}\}$
 \end{center}
Now take a weight $\sigma$ contained in $\tau$, and the unique line $\sigma^{I_{1, K}}$ is given by a non-zero vector which we take as our $w$. Then $w$ is a linear combination of the elements from the above set. This implies, as before that $\phi (w)$ is fixed by $N'_{m_K + 2m-2}$\footnote{Note that in this case $w$ can be chosen as $v$, but as the argument indicates there are other possibilities. This means that, under the assumption $S_- v= S_K  v =0$, the vector $\phi (v)$ is already fixed by the whole group $I_{1, K}$.}.

\smallskip
Note that in each case the group $I_K$ acts as a character on the vector $w$ we find. By repeating the process, we find a non-zero $w \in \pi^{I_{1, K}}$ on which $I_K$ acts as a character, such that $\phi (w)$ is fixed by $N'_{m_K}$. Note that $\phi (w)$ is automatically fixed by $B\cap I_{1, K}$, whence the vector $\phi (w)$ lies in $(\pi')^{I_{1, K}}$.

\smallskip

We proceed to complete the proof. Recall our condition that $\pi$ is supersingular. We apply Proposition \ref{nilpotent of S+} to the vector $w$. We find some $k\geq 1$ such that $S^k _+ w =0$ but $S^{k-1} _+ w \neq 0$. We write the vector $S^{k-1} _+ w$ as $w'$. Then $w'$ is a non-zero $I_{1, K}$-invariant such that
$S_+ w' =0$. This, as we already recorded in \eqref{beta_K v in  Bv: explicit form}, gives that
\begin{center}
$\beta_K w'= -\sum_u u \alpha^{-2} h(\bar{t}^{-1})w' -\sum_u u \alpha^{-1} h(\bar{t}^{-1})w'$
\end{center}
where, in the first sum $u= n(\ast, \varpi^{n_K +2} _E t)$ goes through $(N_{n_K +2}\smallsetminus N_{n_K +3})/N_{n_K +4}$, which we denote by $\mathfrak{N}_1$; in the second sum $u= n(\ast, \varpi^{n_K +1} _E t)$ goes through $(N_{n_K +1}\smallsetminus N_{n_K +2})/N_{n_K +2}$, which we denote by $\mathfrak{N}_2$. As $\phi$ is a $B$-map, we firstly see that
\begin{center}
$\phi(\beta_K w')= -\sum_{u\in \mathfrak{N} _1} u \alpha^{-2} h(\bar{t}^{-1})\phi(w') -\sum_{u\in \mathfrak{N}_2 } u \alpha^{-1} h(\bar{t}^{-1})\phi(w')$
\end{center}
But we also have $\phi (S_+ w')= S_+ \phi (w') =0$, which gives us similarly that
\begin{center}
$\beta_K \phi(w')= -\sum_{u\in \mathfrak{N} _1} u \alpha^{-2} h(\bar{t}^{-1})\phi(w') -\sum_{u\in \mathfrak{N}_2} u \alpha^{-1} h(\bar{t}^{-1})\phi(w')$
\end{center}
We conclude
\begin{center}
$\phi (\beta_K w')= \beta_K \phi(w')$.
\end{center}

Recall that the vector $w'$ and its image $\phi (w')$ are both $I_{1, K}$-invariant. As $G= BI_{1, K} \cup B\beta_K I_{1, K}$ and the non-zero vector $w'$ generates the representation $\pi$, we conclude from the above equality that the $B$-map $\phi$ is indeed a $G$-map.
\end{proof}

\begin{remark}
In early versions of this paper, we didn't find a full proof of Theorem \ref{hom_G= hom_B for ss}. We came to the above proof recently. The operator $S_-$ is an analogue of the element $\Pi=\begin{pmatrix} 0  & 1  \\  \varpi_F & 0
\end{pmatrix}$
in $GL_2 (F)$, but the situation that $S_- v =0$ for some $v\in \pi^{I_{1, K}}$ can always happen, which is not the case for $GL_2 (F)$. This is dealt with in our above \textbf{Case III} argument, and we consider it as a novelty of our work. It would be interesting to see whether it can be adapted to other problems.
\end{remark}

\section*{Acknowledgements}
Most part of this work was done when the author was a postdoc at Einstein Institute of Mathematics, supported by ERC Grant AdG 669655, and our debt owed to the works of Pa$\check{\text{s}}$k$\bar{\text{u}}$nas (\cite{Pas07}) and Hu (\cite{Hu12}) should be clear to the reader. The author is currently supported by Department of Education of Zhejiang Province (Y202351903) and a start-up grant from Tongji Zhejiang College.

\bibliographystyle{amsalpha}
\bibliography{new}

\texttt{Tongji Zhejiang College, Jiaxing, 314051, China}

\emph{E-mail address}: \texttt{xupeng2012@gmail.com}

\end{document}